\documentclass[a4paper]{amsart}
\usepackage{amsmath, amsthm, amsfonts, mathtools, tikz-cd, amssymb}
\usepackage{enumitem}
\usepackage{hyperref}
\usepackage{mathabx}

\title[Finiteness properties and quasi-isometry of group pairs]{Finiteness properties and quasi-isometry of\\ group pairs}
\date{March~6, 2026}


\author[K.~Li]{Kevin Li}
\address{Institut f\"ur Mathematik, Freie Universit\"at Berlin, 14195 Berlin, Germany}
\email{kevin.li@fu-berlin.de}

\author[L.J.~S\'anchez Salda\~na]{Luis Jorge S\'anchez Salda\~na}
\address{Departamento de Matem\'aticas, Facultad de Ciencias, Universidad Nacional Aut\'onoma de M\'exico, 04510 Ciudad de M\'exico, Mexico}
\email{luisjorge@ciencias.unam.mx}

\theoremstyle{definition}
\newtheorem{defn}{Definition}[section]
\newtheorem{ex}[defn]{Example}
\newtheorem{rem}[defn]{Remark}
\newtheorem*{ack}{Acknowledgements}
\theoremstyle{plain}
\newtheorem{lem}[defn]{Lemma}
\newtheorem{prop}[defn]{Proposition}
\newtheorem{thm}[defn]{Theorem}
\newtheorem{cor}[defn]{Corollary}

\numberwithin{equation}{section}
\newcommand{\enum}{\rm{(\roman*)}}
\newcommand{\spann}[1]{{\ensuremath \langle{#1}\rangle}}
\newcommand{\onto}{\twoheadrightarrow}
\renewcommand{\hat}{\widehat}
\newcommand{\sing}{\textup{sing}}

\newcommand{\IN}{\ensuremath{\mathbb{N}}}
\newcommand{\IZ}{\ensuremath{\mathbb{Z}}}
\newcommand{\IR}{\ensuremath{\mathbb{R}}}

\newcommand{\sfF}{\ensuremath{\mathsf{F}}}
\newcommand{\sfFP}{\ensuremath{\mathsf{FP}}}

\newcommand{\calP}{\ensuremath{\mathcal{P}}}
\newcommand{\calQ}{\ensuremath{\mathcal{Q}}}
\newcommand{\calF}{\ensuremath{\mathcal{F}}}
\newcommand{\calFP}{{\ensuremath{\mathcal{F}\spann{\mathcal{P}}}}}
\newcommand{\calFQ}{{\ensuremath{\mathcal{F}\spann{Q}}}}
\newcommand{\OFG}{{\ensuremath{\mathcal{O}_\mathcal{F}G}}}
\newcommand{\OFPG}{{\ensuremath{\mathcal{O}_\calFP G}}}
\newcommand{\EFG}{{\ensuremath{E_\mathcal{F}G}}}
\newcommand{\EFPG}{{\ensuremath{E_\calFP G}}}

\DeclareMathOperator{\hdist}{hdist}
\DeclareMathOperator{\id}{id}
\DeclareMathOperator{\Tor}{Tor}
\DeclareMathOperator{\diam}{diam}
\DeclareMathOperator{\Cone}{Cone}
\DeclareMathOperator{\Map}{Map}

\begin{document}

\begin{abstract}
	We show that the geometric and homological finiteness properties of group pairs are invariant under a suitable notion of quasi-isometry for group pairs.
\end{abstract}

\maketitle

\section{Introduction}

The topic of group cohomology is a rich interplay between algebraic topology, group theory, and homological algebra.
Topological properties of CW-complexes translate, via the classifying space, to properties of groups and these have algebraic counterparts.

Let~$G$ be a discrete group and let~$n\in \IN$.
The group~$G$ is \emph{of type~$\sfF_n$} if there exists a CW-model for~$BG$ with finite $n$-skeleton. Or equivalently, if there exists a $G$-CW-model for~$EG$ with cocompact $n$-skeleton.
The group~$G$ is \emph{of type~$\sfFP_n$} if the trivial $\IZ G$-module~$\IZ$ admits a projective $\IZ G$-resolution that is finitely generated in degrees~$\le n$.
A group is of type~$\sfF_1$ if and only if it is finitely generated, and of type~$\sfF_2$ if and only if it is finitely presented.
For~$n\ge 2$, a group is of type~$\sfF_n$ if and only if it is finitely presented and of type~$\sfFP_n$.
For example, hyperbolic groups are of type~$\sfF_n$ for every~$n\in \IN$.

These topologically and algebraically defined finiteness properties of groups are invariant under the geometric equivalence relation of quasi-isometry and, in fact, inherited by quasi-retracts.
A finitely generated group is considered as a metric space by equipping it with the word metric associated to some finite generating set.

\begin{thm}[Alonso~\cite{Alonso94}]
\label{thm:Alonso_intro}
	Let~$G$ and~$H$ be finitely generated groups and let~$n\in \IN$ with~$n\ge 2$.
	Assume that~$G$ is a quasi-retract of~$H$. 
	Then the following hold:
	\begin{enumerate}[label=\enum]
		\item If~$H$ is of type~$\sfFP_n$, then~$G$ is of type~$\sfFP_n$;
		\item If $H$ is of type~$\sfF_n$, then~$G$ is of type~$\sfF_n$.
	\end{enumerate}
	In particular, if~$G$ and~$H$ are quasi-isometric, then~$G$ is of type~$\sfFP_n$ (resp.\ of type~$\sfF_n$) if and only if~$H$ is of type~$\sfFP_n$ (resp.\ of type~$\sfF_n$).
\end{thm}

\subsection*{Finiteness properties of group pairs}

In many topological situations it is natural and necessary to consider a relative setting of pairs, e.g., for manifolds with boundary.
The cohomology of group pairs was developed systematically, for instance, to treat Poincar\'e duality group pairs~\cite{Bieri-Eckmann78,Alonso91}.
Recently, there is renewed interest in the cohomology of group pairs, most notably in the context of relatively hyperbolic groups and their generalisations~\cite{Manning-Wang,Petrosyan-Sun,MPSV,Patil25}.

Throughout, we will consider only group pairs of the following form.

\begin{defn}
\label{defn:group_pair}
	A \emph{group pair~$(G,\calP)$} consists of a finitely generated group~$G$ and a non-empty finite collection~$\calP$ of subgroups of~$G$.
	(The collection~$\calP$ may contain repetitions.)
	
	We denote by~$G/\calP$ the $G$-set~$\coprod_{P\in \calP} G/P$.
	We will also view~$G/\calP$ as the collection of cosets~$(gP)_{P\in \calP,\, gP\in G/P}$.
	Consider the associated permutation $\IZ G$-module~$\IZ[G/\calP]\coloneqq \bigoplus_{P\in \calP}\bigoplus_{gP\in G/P}\IZ$. Let~$\varepsilon\colon \IZ[G/\calP]\to \IZ$ be the augmentation $\IZ G$-map that is the identity on every summand.
	We denote by~$\Delta_{G/\calP}$ the $\IZ G$-module~$\ker(\varepsilon)$.
\end{defn}

The finiteness properties of group pairs are defined as follows.

\begin{defn}
\label{defn:Fn_pairs}
	Let~$(G,\calP)$ be a group pair and let~$n\in \IN$ with~$n\ge 1$.
		The group pair~$(G,\calP)$ is \emph{of type~$\sfF_n$} if there exists a $G$-CW-pair~$(X,G/\calP)$ such that every cell in~$X\smallsetminus G/\calP$ has trivial stabiliser, $X$ is contractible, and the $n$-skeleton of~$X$ is cocompact.
		
		The group pair~$(G,\calP)$ is \emph{of type~$\sfFP_n$} if the $\IZ G$-module~$\Delta_{G/\calP}$ is of type~$\sfFP_{n-1}$, i.e., there exists a projective $\IZ G$-resolution $\cdots\to P_1\to P_0\to \Delta_{G/\calP}\to 0$ such that for every~$i=0,\ldots,n-1$, the $\IZ G$-module~$P_i$ is finitely generated.
\end{defn}

The finiteness properties of a group~$G$ are equivalent to the finiteness properties of the group pair~$(G,\{1\})$, where~$\{1\}$ is the collection consisting only of the trivial subgroup.
A group pair is of type~$\sfF_2$ if and only if it is relatively finitely presented (Definition~\ref{defn:rel fin pres}).
For~$n\ge 2$, a group pair is of type~$\sfF_n$ if and only if it is relatively finitely presented and of type~$\sfFP_n$ (Theorem~\ref{thm:Eilenberg-Ganea}).
For example, relatively hyperbolic group pairs are of type~$\sfF_n$ for every~$n\ge 1$~\cite{MPP,Patil25}. 

\subsection*{Quasi-isometry of group pairs}
A quasi-isometry of group pairs is a quasi-isometry of groups that ``coarsely preserves" the cosets of subgroups in the collections.
A function~$f\colon X\to Y$ between metric spaces is \emph{$(L,C)$-Lipschitz} if for all~$x,x'\in X$, we have $d_Y(f(x),f(x'))\le L\cdot d_X(x,x')+C$.

\begin{defn}
\label{defn:strong QI}
	Let~$(G,\calP)$ and~$(H,\calQ)$ be group pairs.
	We equip~$G$ and~$H$ with word metrics associated to some finite generating sets.
	Let~$L,C,M\in \IR$ with~$L\ge 1$, $C\ge 0$, and~$M\ge 0$.
	\begin{itemize}
		\item An \emph{$(L,C,M)$-Lipschitz map of pairs $f=(f_1,f_2)\colon (G,\calP)\to (H,\calQ)$} consists of an $(L,C)$-Lipschitz map~$f_1\colon G\to H$ and a function~$f_2\colon G/\calP\to H/\calQ$ such that for every~$A\in G/\calP$, we have
		\[
			\hdist_H\bigl(f_1(A),f_2(A)\bigr) <M,
		\]
		where~$\hdist_H$ denotes the Hausdorff distance between subsets of~$H$.
		\item An \emph{$(L,C,M)$-quasi-retraction of pairs~$(f,r)$} consists of $(L,C,M)$-Lipschitz maps of pairs~$f\colon (G,\calP)\to (H,\calQ)$ and~$r\colon (H,\calQ)\to (G,\calP)$ such that for all~$g\in G$, we have~$d_G\bigl(r_1\circ f_1(g),g\bigr)\le C$, and~$r_2\circ f_2=\id_{G/\calP}$.
		\item An $(L,C,M)$-Lipschitz map~$f\colon (G,\calP)\to (H,\calQ)$ is a \emph{strong $(L,C,M)$-quasi-isometry of pairs} if there exists an $(L,C,M)$-Lipschitz map of pairs $r\colon (H,\calQ)\to (G,\calP)$ such that~$(f,r)$ and~$(r,f)$ are $(L,C,M)$-quasi-retrac\-tions of pairs.
	\end{itemize}
	We say that~$(G,\calP)$ is a \emph{quasi-retract} of~$(H,\calQ)$ (resp.\ \emph{strongly quasi-isometric} to~$(H,\calQ)$) if there exist~$L,C,M\in \IR$ with~$L\ge 1$, $C\ge 0$, and~$M\ge 0$ such that~$(G,\calP)$ is an $(L,C,M)$-quasi-retract of~$(H,\calQ)$ (resp.\ strongly $(L,C,M)$-quasi-isometric to~$(H,\calQ)$).
\end{defn}

We relate Definition~\ref{defn:strong QI} to other notions of quasi-isometry for group pairs in the literature~\cite{HMPSS_survey} in Section~\ref{sec:QI pairs}.
This provides many examples of quasi-retractions and strong quasi-isometries of group pairs (Example~\ref{ex:QI pairs}).
Several classical geometric invariants of groups, such as the number of ends and the Dehn function, have relative versions for group pairs and these were recently shown to be invariant under (strong) quasi-isometry of group pairs~\cite{MPSS_ends,HMPSS}.

\subsection*{Quasi-isometry invariance}
Our main result is that the finiteness properties of group pairs are inherited by quasi-retracts and hence invariant under strong quasi-isometry. 
This provides a relative version of Theorem~\ref{thm:Alonso_intro} for group pairs.

\begin{thm}[Theorem~\ref{thm:FPn_QI} and Corollary~\ref{cor:Fn_QI}]
\label{thm:main_intro}
	Let~$(G,\calP)$ and~$(H,\calQ)$ be group pairs and let~$n\in \IN$ with~$n\ge 2$.
	Assume that~$(G,\calP)$ is a quasi-retract of~$(H,\calQ)$.
	Then the following hold:
	\begin{enumerate}[label=\enum]
		\item If~$(H,\calQ)$ is of type~$\sfFP_n$, then~$(G,\calP)$ is of type~$\sfFP_n$;
		\item If~$(H,\calQ)$ is of type~$\sfF_n$, then~$(G,\calP)$ is of type~$\sfF_n$.
	\end{enumerate}
	In particular, if~$(G,\calP)$ and~$(H,\calQ)$ are strongly quasi-isometric, then~$(G,\calP)$ is of type~$\sfFP_n$ (resp.\ of type~$\sfF_n$) if and only if~$(H,\calQ)$ is of type~$\sfFP_n$ (resp.\ of type~$\sfF_n$).
\end{thm}

The strategy to prove Theorem~\ref{thm:main_intro} follows Alonso's proof of Theorem~\ref{thm:Alonso_intro} but requires subtle modifications.
These are needed because, unlike the Cayley graph, the coned-off Cayley graph (Definition~\ref{defn:coned-off Cayley}) is not locally finite and therefore the Rips complex is not cocompact and there are infinitely many orbits of loops of a fixed length.
We circumvent this problem by using variations that contain at most one cone vertex. 

For part~(i), we first establish a relative version of Brown's criterion (Theorem~\ref{thm:Brown_criterion_space}), which may be of independent interest.
We then apply the criterion to the ``unicone Rips complex" (Definition~\ref{defn:unicone Rips}), which is a cocompact subcomplex of the Rips complex on the coned-off Cayley graph.
For part~(ii), we prove that relative finite presentability is inherited by quasi-retracts (Theorem~\ref{thm:fin pres_QI}) using ``unicone loops" (Definition~\ref{defn:unicone loop}) in the coned-off Cayley graph.
Together with part~(i), this yields part~(ii) because type~$\sfF_n$ is equivalent to relative finite presentability and type~$\sfFP_n$ (Theorem~\ref{thm:Eilenberg-Ganea}).

There is a weaker notion of (not necessarily strong) quasi-isometry for group pairs (Definition~\ref{defn:QI pairs}).
We do not know if Theorem~\ref{thm:main_intro} is true more generally for (not necessarily strongly) quasi-isometric group pairs.
It was asked in~\cite[Question~1.2]{HMPSS} whether relative finite presentability is invariant under quasi-isometry of group pairs.

\subsection*{Bredon finiteness properties for families}
Another well-studied notion of relative finiteness properties is given by the Bredon finiteness properties of a group relative to a family of subgroups (Definition~\ref{defn:Bredon}).
These are defined in terms of classifying spaces for families and projective resolutions over the orbit category.
If the collection~$\calP$ is malnormal (Definition~\ref{defn:reduced}), then the finiteness properties of the group pair~$(G,\calP)$ are equivalent to the Bredon finiteness properties of~$G$ relative to the smallest family~$\calFP$ containing~$\calP$ (Proposition~\ref{prop:Bredon_malnormal}).
As a consequence of Theorem~\ref{thm:main_intro}, the Bredon finiteness properties~$\calFP$-$\sfFP_n$ and~$\calFP$-$\sfF_n$ for a malnormal collection~$\calP$ are inherited by quasi-retracts and hence invariant under strong quasi-isometry of group pairs. 

\begin{thm}[Corollary~\ref{cor:Bredon_QI}]
	Let~$(G,\calP)$ and~$(H,\calQ)$ be group pairs, where~$\calP$ and~$\calQ$ are malnormal collections consisting of distinct subgroups.
	Let~$n\in \IN$ with~$n\ge 2$.
	Assume that~$(G,\calP)$ is a quasi-retract of~$(H,\calQ)$.
	Then the following hold:
	\begin{enumerate}[label=\enum]
		\item If~$H$ is of type~$\calFQ$-$\sfFP_n$, then~$G$ is of type~$\calFP$-$\sfFP_n$;
		\item If~$H$ is of type~$\calFQ$-$\sfF_n$, then~$G$ is of type~$\calFP$-$\sfF_n$.
	\end{enumerate}
	In particular, if~$(G,\calP)$ and~$(H,\calQ)$ are strongly quasi-isometric, then~$(G,\calP)$ is of type~$\calFP$-$\sfFP_n$ (resp.\ of type~$\calFP$-$\sfF_n$) if and only if~$(H,\calQ)$ is of type~$\calFQ$-$\sfFP_n$ (resp.\ of type~$\calFQ$-$\sfF_n$).
\end{thm}

\begin{ack}
	Parts of this work were conducted during the authors' stays at the University of Oxford and the Ohio State University.
	We thank these institutions for their hospitality.
	The first author was partially supported by the SFB~1085 \emph{Higher Invariants} (Universit\"at Regensburg, funded by the DFG). 
	The second author is grateful for the financial support of the DGAPA-UNAM grant PAPIIT~IN102426.
\end{ack} 

\section{Quasi-isometry of group pairs}
\label{sec:QI pairs}

We briefly review some results on the geometry of group pairs.
For more background, we refer to~\cite{HMPSS,HMPSS_survey}.

The basic properties of being finitely generated and being finitely presented for groups admit natural extensions to group pairs.

\begin{defn}
\label{defn:rel fin pres}
	Let~$(G,\calP)$ be a group pair and let~$S$ be a subset of~$G$.
	The inclusions~$S\subset G$ and~$P\subset G$ for~$P\in \calP$ induce a group homomorphism
	\begin{equation}
	\label{eqn:rel gen}
		\varphi\colon F\coloneqq F(S)\ast \bigast_{P\in \calP}P\to G,
	\end{equation}
	where~$F(S)$ is the free group on~$S$.
	The set~$S$ is a \emph{relative generating set of~$(G,\calP)$} if the group homomorphism~$\varphi$ is surjective.
	A \emph{relative presentation~$\spann{S,\calP\mid R}$ of~$(G,\calP)$} consists of a relative generating set~$S$ and a subset~$R$ of~$F$ that normally generates~$\ker(\varphi)$ in~$F$.
	A relative presentation~$\spann{S,\calP\mid R}$ is \emph{finite} if the sets~$S$ and~$R$ are finite.
	A group pair~$(G,\calP)$ is \emph{relatively finitely presented} if there exists a finite relative presentation of~$(G,\calP)$.
\end{defn}

The geometry of a group is reflected by its Cayley graph.
For group pairs, this role is played by the coned-off Cayley graph.
Recall that~$G/\calP$ denotes the $G$-set~$\coprod_{P\in \calP}G/P$.

\begin{defn}
\label{defn:coned-off Cayley}
	Let~$(G,\calP)$ be a group pair and let~$S$ be a finite relative generating set of~$(G,\calP)$.
	The \emph{coned-off Cayley graph~$\hat \Gamma(G,\calP,S)$} is the graph with vertex set~$G\sqcup G/\calP$ and two types of edges: $\{g,g'\}$ for all~$g,g'\in G$ with~$g^{-1}g'\in S$, and $\{g,g'P\}$ for all~$g\in G$ and~$g'P\in G/\calP$ with~$g\in g'P$.
\end{defn}

A finitely generated group~$G$ is finitely presented if and only if by attaching finitely many $G$-orbits of $2$-cells to the Cayley graph of~$G$, one can obtain a simply-connected space.
The corresponding result for group pairs is as follows.

\begin{prop}[{\cite[Proposition~4.9]{HMPSS}}]
\label{prop:coned-off Cayley_pi1}
	Let~$(G,\calP)$ be a group pair and let~$S$ be a relative generating set of~$(G,\calP)$.
	Let~$\varphi\colon F\coloneqq F(S)\ast \bigast_{P\in \calP} P\to G$ be the canonical group homomorphism~\eqref{eqn:rel gen}.
	Then the graph~$\hat \Gamma(G,\calP,S)$ is connected and there is an isomorphism of groups 
	\[
		\pi_1(\hat \Gamma(G,\calP,S))\cong \ker(\varphi).
	\]
	In particular, the kernel~$\ker(\varphi)$ is normally finitely generated in~$F$ if and only if there exist finitely many loops in~$\hat \Gamma(G,\calP,S)$ such that the $2$-complex obtained from~$\hat \Gamma(G,\calP,S)$ by attaching a $2$-cell to each of these loops and each of their $G$-translates is simply-connected.
\end{prop}

We will mostly consider the coned-off Cayley graph~$\hat \Gamma(G,\calP,S)$ when~$S$ is a finite (non-relative) generating set of~$G$.
Then the full subgraph of~$\hat \Gamma(G,\calP,S)$ spanned by the vertices in~$G$ is the Cayley graph~$\Gamma(G,S)$.
Note that if~$S$ and~$S'$ are finite generating sets of~$G$, then~$\hat \Gamma(G,\calP,S)$ and~$\hat \Gamma(G,\calP,S')$ are quasi-isometric and we will sometimes just write~$\hat \Gamma(G,\calP)$.

A Lipschitz map of group pairs (Definition~\ref{defn:strong QI}) extends (non-canonically) to a Lipschitz map between coned-off Cayley graphs.

\begin{lem}[{\cite[Lemma~5.9]{HMPSS}}]
\label{lem:coned-off Cayley_functorial}
	Let~$(G,\calP)$ and~$(H,\calQ)$ be group pairs.
	Let~$f=(f_1,f_2)\colon (G,\calP)\to (H,\calQ)$ be an $(L,C,M)$-Lipschitz map of pairs. 
	We denote~$\hat f\coloneqq f_1\sqcup f_2\colon G\sqcup G/\calP\to H\sqcup H/\calQ$ and set $\hat L\coloneqq \max\{L+C,M+1\}$.
	Then there exists a cellular map of $1$-dimensional CW-complexes
	\[
		\hat f_*\colon \hat \Gamma(G,\calP)\to \hat \Gamma(H,\calQ)
	\]
	given on vertices by~$\hat f$ and satisfying the following:
		Every (oriented) edge~$[v,w]$ in~$\hat \Gamma(G,\calP)$ is mapped to an edge path~$\hat f_*[v,w]$ in~$\hat \Gamma(H,\calQ)$ from~$\hat f(v)$ to~$\hat f(w)$ of length~$\le \hat L$.
	Moreover, the number of vertices of~$\hat f_*[v,w]$ in~$H/\calQ$ equals the number of vertices of~$[v,w]$ in~$G/\calP$.
	\begin{proof}
		We fix an orientation of the edges in~$\hat \Gamma(G,\calP)$.
		For an oriented edge~$[v,w]$ in~$\hat \Gamma(G,\calP)$, there are three cases:
		If~$v,w\in G$, we take~$\hat f_*[v,w]$ to be a geodesic in the Cayley graph~$\Gamma(H)$ from~$\hat f(v)$ to~$\hat f(w)$.
			Since~$f_1$ is $(L,C)$-Lipschitz, the length of~$\hat f_*[v,w]$ is bounded by~$L+C$.
		
		If~$v\in G$ and~$w\in G/\calP$, then~$v$ is an element of the coset~$w$.
			Since~$f$ is an~$(L,C,M)$-Lipschitz map of pairs, there exists an element~$u$ of the coset~$f_2(w)$ such that~$d_H(f_1(v),u)\le M$.
			We take~$\hat f_*[v,w]$ to be the concatenation of a geodesic in~$\Gamma(H)$ from~$f_1(v)$ to~$u$ and the edge from~$u$ to the vertex~$f_2(w)$.
			Then the length of~$\hat f_*[v,w]$ is bounded by~$M+1$.
		
		If~$v\in G/\calP$ and~$w\in G$, we take~$\hat f_*[v,w]$ to be defined analogously to the previous case.
	\end{proof}
\end{lem}

We relate the notion of strong quasi-isometry of group pairs (Definition~\ref{defn:strong QI}) to other notions in the literature~\cite{HMPSS_survey}.

\begin{defn}
\label{defn:QI pairs}
	Let~$(G,\calP)$ and~$(H,\calQ)$ be group pairs. We equip~$G$ and~$H$ with word metrics with respect to some finite generating sets.
	Let~$L,C,M\in \IR$ with $L\ge 1$, $C\ge 0$, and~$M\ge 0$.
	For a function~$q\colon G\to H$, we define the relation
	\[
		\dot{q}\coloneqq \bigl\{(A,B)\in G/\calP\times H/\calQ\bigm| \hdist_H(q(A),B)<M\bigr\},
	\]
	where~$\hdist_H$ denotes the Hausdorff distance between subsets of~$H$.
		An \emph{$(L,C,M)$-quasi-isometry of pairs~$q\colon (G,\calP)\to (H,\calQ)$} is an $(L,C)$-quasi-isometry~$q\colon G\to H$ such that the projections~$\dot{q}\to G/\calP$ and~$\dot{q}\to H/\calQ$ are surjective.
\end{defn}

One also considers quasi-isometries of pairs~$q$ for which the projections from~$\dot q$ are additionally bijective~\cite[Section~5]{HMPSS}.
Strong quasi-isometry of pairs (Definition~\ref{defn:strong QI}) fits between the two notions.

\begin{lem}
	Let~$(G,\calP)$ and~$(H,\calQ)$ be group pairs.
	The following hold:
	\begin{enumerate}[label=\enum]
		\item If~$q\colon (G,\calP)\to (H,\calQ)$ is a quasi-isometry of pairs such that the projections~$\dot q\to G/\calP$ and~$\dot q\to H/\calQ$ are bijective, then~$(G,\calP)$ and~$(H,\calQ)$ are strongly quasi-isometric;
		\item If~$(G,\calP)$ and~$(H,\calQ)$ are strongly quasi-isometric, then~$(G,\calP)$ and~$(H,\calQ)$ are quasi-isometric.
	\end{enumerate}
	\begin{proof}
		(i) Let~$q\colon (G,\calP)\to (H,\calQ)$ be an $(L,C,M)$-quasi-isometry such that the projections~$\dot q\to G/\calP$ and~$\dot q\to H/\calQ$ are bijective.
		Let~$f_1\coloneqq q$ and let \mbox{$f_2\colon G/\calP\to H/\calQ$} be the bijection given by the relation~$\dot q$.
		Then~$f=(f_1,f_2)\colon (G,\calP)\to (H,\calQ)$ is an $(L,C,M)$-Lipschitz map of pairs.
		Let~$r_1\colon H\to G$ be an $(L,C)$-quasi-inverse of~$q$ and let~$r_2\colon H/\calQ\to G/\calP$ be the inverse of~$f_2$.
		Then~$r=(r_1,r_2)\colon (H,\calQ)\to (G,\calP)$ is a Lipschitz map of pairs.
		Indeed, for every~$B\in H/\calQ$, we have
		\begin{align*}
			\hdist_G\bigl(r_1(B),r_2(B)\bigr)
			&\le \hdist_G\bigl(r_1(B),r_1\circ q_1(r_2(B))\bigr)+C
			\\
			&\le L\cdot \hdist_H\bigl(B,q_1(r_2(B))\bigr) +2C
			\\
			&= L\cdot \hdist_H\bigl(q_2\circ r_2(B),q_1(r_2(B))\bigr) +2C
			\\
			&< L\cdot M+2C.
		\end{align*}
		By construction, $(f,r)$ and~$(r,f)$ are quasi-retractions of pairs.
		
		(ii) Let~$f=(f_1,f_2)\colon (G,\calP)\to (H,\calQ)$ be a strong $(L,C,M)$-quasi-isometry of pairs.
		Then~$q\coloneqq f_1\colon (G,\calP)\to (H,\calQ)$ is an $(L,C,M)$-quasi-isometry of pairs.
		Indeed, since~$f_2$ is bijective, the projections~$\dot q\to G/\calP$ and~$\dot q\to H/\calQ$ are surjective.
 	\end{proof}
\end{lem}

We list some examples of quasi-isometries of group pairs.

\begin{ex}
\label{ex:QI pairs}
	For the following examples, we assume that the collections consist of distinct subgroups.
	\begin{enumerate}[label=\enum]
		\item Let~$(G,\calP)$ be a group pair and let~$H$ be a finite-index subgroup of~$G$. 
		Then there exists an explicit collection~$\calQ$ of subgroups of~$H$ such that the inclusion~$q\colon H\to G$ is a quasi-isometry of pairs~$q\colon (H,\calQ)\to (G,\calP)$~\cite[Proposition~2.13]{MPSS_ends}.
		\item Let~$(G,\calP)$ and~$(H,\calQ)$ be relatively hyperbolic group pairs, where the collections~$\calP$ and~$\calQ$ consist of finitely generated non-relatively hyperbolic groups.
		Then every quasi-isometry~$q\colon G\to H$ of groups is a quasi-isometry of pairs~$q\colon (G,\calP)\to (H,\calQ)$~\cite[Theorem~4.1]{BDM}.
		\item Let~$(G,\calP)$ be a group pair, where the collection~$\calP$ is quasi-isometrically characteristic \cite[Definition~2.3]{MPSS_ends}, and let~$H$ be a finitely generated group.
		If~$q\colon G\to H$ is a quasi-isometry of groups, then there exists a quasi-isometrically characteristic collection~$\calQ$ of subgroups of~$H$ such that $q\colon (G,\calP)\to (H,\calQ)$ is a quasi-isometry of pairs~\cite[Theorem~1.1]{MPSS_ends}.
		For many concrete examples of quasi-isometrically characteristic collections, we refer to~\cite[Section~4]{HMPSS_survey}.
	\end{enumerate}
\end{ex}

We conclude this section with a sufficient condition for a quasi-isometry of pairs~$q$ to have bijective projections from~$\dot q$.

\begin{defn}
\label{defn:reduced}
	Let~$(G,\calP)$ be a group pair.
		The collection~$\calP$ is \emph{reduced} if for all~$P_1,P_2\in \calP$ and~$g\in G$, the following holds: If~$P_1\cap gP_2g^{-1}$ has finite index in both~$P_1$ and~$gP_2g^{-1}$, then~$P_1=P_2$ and~$g\in P_1$.
		The collection~$\calP$ is \emph{almost malnormal} (resp.\ \emph{malnormal}) if for all~$P_1,P_2\in \calP$ and~$g\in G$, the following holds: If~$P_1\cap gP_2g^{-1}$ is infinite (resp.\ non-trivial), then~$P_1=P_2$ and~$g\in P_1$.
\end{defn}

Clearly, every malnormal collection is almost malnormal and every almost malnormal collection of infinite subgroups is reduced.

\begin{prop}[{\cite[Proposition~5.12]{HMPSS}}]
	Let~$(G,\calP)$ and~$(H,\calQ)$ be group pairs.
	If the collections~$\calP$ and~$\calQ$ are reduced and consist of distinct subgroups, then for every quasi-isometry of pairs $q\colon (G,\calP)\to (H,\calQ)$, the projections~$\dot q\to G/\calP$ and~$\dot q\to H/\calQ$ are bijective.
\end{prop}

\section{Finiteness properties of group pairs}

In this section, we establish a relative version of Brown's criterion~\cite[Theorem~2.2]{Brown87} in Theorem~\ref{thm:Brown_criterion_space}.
It characterises the homological finiteness properties of group pairs in terms of resolutions and the essential acyclicity of appropriate filtrations.
Throughout, we consider left modules and left actions.

Recall that an $R$-module~$M$ is \emph{of type~$\sfFP_n$} if there exists a projective $R$-resolution of~$M$ that is finitely generated in degrees~$\le n$.
A group pair~$(G,\calP)$ is of type~$\sfFP_n$ if the $\IZ G$-module~$\Delta_{G/\calP}\coloneqq \ker(\varepsilon\colon \IZ[G/\calP]\to \IZ)$ is of type~$\sfFP_{n-1}$ (Definition~\ref{defn:Fn_pairs}).
To deduce the finiteness properties of~$\Delta_{G/\calP}$, one can use $\IZ G$-resolutions of~$\IZ$ that are projective except in degree~$0$. 

\begin{lem}
\label{lem:0-unfree_resolutions}
	Let~$(G,\calP)$ be a group pair and let~$n\in \IN$ with~$n\ge 1$.
	The following are equivalent:
	\begin{enumerate}[label=\enum]
		\item The group pair~$(G,\calP)$ is of type~$\sfFP_n$;
		\item The $\IZ G$-module~$\Delta_{G/\calP}$ is finitely generated and for every~$k=1,\ldots,n-1$ and every exact sequence of $\IZ G$-modules
		\[
			P_k\xrightarrow{\partial_k} P_{k-1}\to \cdots\to P_1\xrightarrow{\partial_1} P_0\oplus \IZ[G/\calP]\xrightarrow{\eta+ \varepsilon} \IZ\to 0,
		\]
		where the $\IZ G$-module~$P_i$ is finitely generated projective for every~$i=0,\ldots,k$, we have that the $\IZ G$-module~$\ker(\partial_k)$ is finitely generated;
		\item There exists an exact sequence of $\IZ G$-modules
		\[
			L_n\to L_{n-1}\to \cdots\to L_1\to L_0\oplus \IZ[G/\calP]\xrightarrow{\eta+ \varepsilon} \IZ\to 0,
		\]
		where the $\IZ G$-module~$L_i$ is finitely generated free for every~$i=0,\ldots,n$.
	\end{enumerate}
	\begin{proof}
		Suppose that~(i) holds and let a sequence as in~(ii) be given.
		Applying the snake lemma to the diagram
		\[\begin{tikzcd}
			0\ar{r}
			& \IZ[G/\calP]\ar{r}\ar{d}{\varepsilon}
			& P_0\oplus \IZ[G/\calP]\ar{r}\ar{d}{\eta+ \varepsilon}
			& P_0\ar{r}\ar{d}
			& 0
			\\
			0\ar{r}
			& \IZ\ar{r}{=}
			& \IZ\ar{r}
			& 0\ar{r}
			& 0
		\end{tikzcd}\]
		yields an isomorphism of $\IZ G$-modules $\ker(\eta+ \varepsilon)\cong P_0\oplus \Delta_{G/\calP}$.
		Since~$\Delta_{G/\calP}$ is of type~$\sfFP_{n-1}$ by assumption and~$P_0$ is finitely generated projective, the $\IZ G$-module~$\ker(\eta+ \varepsilon)$ is of type~$\sfFP_{n-1}$~\cite[Proposition~1.4]{Bieri81}.
		Considering the exact sequence
		\[
			P_k\xrightarrow{\partial_k} P_{k-1}\to \cdots\to P_1\xrightarrow{\partial_1} \ker(\eta+ \varepsilon)\to 0
		\]
		and applying~\cite[Proposition~VIII.4.3]{Brown82} yields the claim.
		
		Suppose that~(ii) holds.
		Then an exact sequence as in~(iii) can be built inductively.
		
		Suppose that~(iii) holds.
		Considering the exact sequence
		\[
			L_n\to L_{n-1}\to \cdots\to L_1\to \ker(\eta+ \varepsilon)\to 0
		\]
		shows that the $\IZ G$-module~$\ker(\eta+ \varepsilon)$ is of type~$\sfFP_{n-1}$.
		There exists an isomorphism of $\IZ G$-modules~$\ker(\eta+ \varepsilon)\cong L_0\oplus \Delta_{G/\calP}$, as in the first part of the proof.
		Since~$L_0$ is finitely generated free, it follows that~$\Delta_{G/\calP}$ is of type~$\sfFP_{n-1}$~\cite[Proposition~1.4]{Bieri81}.
	\end{proof}
\end{lem}

We will want to use resolutions that might not be finitely generated, but admit filtrations by finitely generated (not necessarily exact) subcomplexes.

\begin{defn}
Let~$R$ be a ring and let~$M$ be an $R$-module.
An \emph{$R$-chain complex augmented over~$M$} consists of an $R$-chain complex~$(C_*,\partial_*)$ concentrated in degrees~$\ge 0$ together with a surjective $R$-map $\eta\colon C_0\onto M$ such that~$\eta\circ \partial_1=0$.
(The augmentation map~$\eta$ will often be left implicit.)
Let~$C_*$ be an $R$-chain complex augmented over~$M$.
An $R$-subcomplex~$C_*'$ of~$C_*$ is \emph{augmented} if the restricted map~$\eta|_{C_0'}\colon C_0'\to M$ is surjective.
The augmentation map~$\eta$ induces a surjection~$H_0(\eta)\colon H_0(C_*)\onto M$.
For~$i\in \IN$, we define the \emph{reduced homology of~$C_*$} as
\[
	\widetilde{H}_i(C_*)\coloneqq \begin{cases}
		H_i(C_*) & \text{if } i\ge 1; \\
		\ker(H_0(\eta)) & \text{if } i=0.
	\end{cases}
\]

Let~$D$ be a directed set (which will often be left implicit).
A \emph{filtration} of an $R$-chain complex~$C_*$ is a directed system of $R$-subcomplexes~$(C^\alpha_*)_{\alpha\in D}$ along inclusions such that $C_*=\varinjlim_{\alpha\in D}C^\alpha_*$.
A filtration of a $G$-CW-complex by $G$-subcomplexes is defined analogously.

A directed system of groups~$(H_\alpha)_{\alpha\in D}$ is \emph{essentially trivial} if for every~$\alpha\in D$ there exists~$\beta\in D$ with~$\beta\ge \alpha$ such that the structure map~$H_\alpha\to H_\beta$ is trivial.
\end{defn}

First, we characterise finiteness properties of modules using free partial resolutions and the essential acyclicity of filtrations by finitely generated subcomplexes.

\begin{lem}
\label{lem:Bieri-Eckmann}
	Let~$R$ be a ring, let~$M$ be an $R$-module, and let~$n\in \IN$.
	Let~$L_*$ be a free $R$-chain complex that is augmented over~$M$.
	Let~$(L_*^\alpha)_\alpha$ be a filtration of~$L_*$ by augmented $R$-subcomplexes.
	Suppose that the following hold:
	\begin{enumerate}[label=\enum]
		\item For every~$i=1,\ldots,n-1$, we have $H_i(L_*)=0$, and the augmentation induces an isomorphism $H_0(L_*)\cong M$;
		\item For every~$\alpha$ and every~$i=0,\ldots,n$, the $R$-module~$L^\alpha_i$ is finitely generated free.
	\end{enumerate}
	Then the $R$-module~$M$ is of type~$\sfFP_n$ if and only if for every~$i=0,\ldots,n-1$, the directed system of groups~$(\widetilde{H}_i(L^\alpha_*))_\alpha$ is essentially trivial.
	\begin{proof}
		The claim is true for~$n=0$.
		Let~$n\ge 1$.
		By~\cite[Proposition~1.2 and Lemma~1.1]{Bieri-Eckmann74}, the $R$-module~$M$ is of type~$\sfFP_n$ if and only if for every index set~$J$ and every~$i=1,\ldots,n-1$, we have
		$\Tor^R_i(\prod_J R,M)=0$ and the natural map $\Tor^R_0(\prod_J R,M)\to \prod_J M$ is an isomorphism.
		For~$i\le n-1$, there are isomorphisms of $R$-modules
		\begin{align*}
			\Tor^R_i(\prod_J R,M)
			&\cong H_i\bigl((\prod_J R)\otimes_R L_*\bigr) 
			\cong H_i\bigl((\prod_J R)\otimes_R \varinjlim_\alpha L^\alpha_*\bigr)
			\\
			&\cong H_i\bigl(\varinjlim_\alpha (\prod_J R)\otimes_R L^\alpha_*\bigr)
			\cong \varinjlim_\alpha H_i\bigl((\prod_J R)\otimes_R L^\alpha_*\bigr)
			\\
			&\cong \varinjlim_\alpha H_i(\prod_J L^\alpha_*)
			\cong \varinjlim_\alpha \prod_J H_i(L^\alpha_*)
		\end{align*}
		where the first isomorphism uses assumption~(i) and the second to last isomorphism uses assumption~(ii).
		Moreover, for~$i=0$, there is a short exact sequence of $R$-modules
		\[
			0\to 
			\varinjlim_\alpha \prod_J \widetilde{H}_0(L^\alpha_*)\to
			\varinjlim_\alpha \prod_J H_0(L^\alpha_*)\to
			\prod_J M\to
			0.
		\]
		By~\cite[Lemma~2.1]{Brown87}, a directed system of groups~$(H_\alpha)_\alpha$ is essentially trivial if and only if for every index set~$J$, we have $\varinjlim_\alpha \prod_J H_\alpha=0$.
		This finishes the proof. 
	\end{proof}
\end{lem}

Next, over the group ring, we relax the partial resolutions from being free to consisting of permutation modules.

\begin{prop}
\label{prop:Brown_criterion_chain}
	Let~$G$ be a group, let~$M$ be a $\IZ G$-module, and let~$n\in \IN$.
	Let~$C_*$ be a $\IZ G$-chain complex that is augmented over~$M$.
	Let~$(C^\alpha_*)_\alpha$ be a filtration of~$C_*$ by augmented $\IZ G$-subcomplexes.
	Suppose that the following hold:
	\begin{enumerate}[label=\enum]
		\item For every~$i=1,\ldots,n-1$, we have $H_i(C_*)=0$, and the augmentation induces an isomorphism $H_0(C_*)\cong M$;
		\item For every~$i\in \IN$, we have $C_i=\bigoplus_{\sigma\in \Sigma_i} \IZ[G/G_\sigma]$ for some set~$\Sigma_i$ and subgroups~$G_\sigma$ of~$G$;
		\item For every~$i\in \IN$ and every~$\sigma\in \Sigma_i$, the group~$G_\sigma$ is of type~$\sfFP_{n-i}$;
		\item For every~$\alpha$ and every~$i=0,\ldots,n$, we have $C^\alpha_i=\bigoplus_{\sigma\in \Omega_i} \IZ[G/G_\sigma]$, where~$\Omega_i$ is a finite subset of~$\Sigma_i$.
	\end{enumerate}
	Then the $\IZ G$-module~$M$ is of type~$\sfFP_n$ if and only if for every~$i=0,\ldots,n-1$, the directed system of groups~$(\widetilde{H}_i(C^\alpha_*))_\alpha$ is essentially trivial.
	\begin{proof}
		For~$i\in \IN$ and~$\sigma\in \Sigma_i$, let~$L^\sigma_*$ be a free $\IZ G_\sigma$-resolution of~$\IZ$ that is finitely generated in degrees~$\le n-i$.
		By a standard construction~\cite[Proposition~1.1]{Brown87}, there exists a free $\IZ G$-chain complex~$L_*$ such that for every~$k\in \IN$, we have
		\[
			L_k=\bigoplus_{i+j=k}\bigoplus_{\sigma\in \Sigma_i} \IZ G\otimes_{\IZ G_\sigma} L^\sigma_j
		\]
		together with a $\IZ G$-chain map $f\colon L_*\to C_*$ that is a quasi-isomorphism (i.e., that induces isomorphisms in homology).
		The filtration~$(C^\alpha_*)_\alpha$ of~$C_*$ gives rise to a filtration~$(L^\alpha_*)_\alpha$ of~$L_*$, where for every~$\alpha$ and every~$k\in \IN$, we have
		\[
			L^\alpha_k=\bigoplus_{i+j=k}\bigoplus_{\sigma\in \Omega_i} \IZ G\otimes_{\IZ G_\sigma} L^\sigma_j.
		\]
		In particular, for every~$k\le n$, the $\IZ G$-module~$L^\alpha_k$ is finitely generated free.
		Moreover, the restriction $f|_{L^\alpha_*}\colon L^\alpha_*\to C^\alpha_*$ is a quasi-isomorphism.
		Hence, for every~$i\in \IN$, the directed systems of groups~$(\widetilde{H}_i(L^\alpha_*))_\alpha$ and~$(\widetilde{H}_i(C^\alpha_*))_\alpha$ are isomorphic and, in particular, one is essentially trivial if and only if the other one is.
		Then the claim follows by applying Lemma~\ref{lem:Bieri-Eckmann} to the free $\IZ G$-chain complex~$L_*$ augmented over~$M$ and the filtration~$(L^\alpha_*)_\alpha$ of~$L_*$.
	\end{proof}
\end{prop}

Finally, we characterise the homological finiteness properties of group pairs using cellular chain complexes of $G$-CW-complexes.
We denote the $n$-skeleton of a $G$-CW-complex~$X$ by~$X^{(n)}$.

\begin{thm}
\label{thm:Brown_criterion_space}
	Let~$(G,\calP)$ be a group pair and let~$n\in \IN$ with~$n\ge 1$.
	Let~$(X,G/\calP)$ be a $G$-CW-pair.
	Let~$(X_\alpha)_\alpha$ be a filtration of~$X$ by $G$-subcomplexes.
	Suppose that the following hold:
	\begin{enumerate}[label=\enum]
		\item $X$ is $(n-1)$-acyclic;
		\item For every cell~$\sigma$ in~$X\smallsetminus G/\calP$, the stabiliser~$G_\sigma$ is of type~$\sfFP_{n-\dim(\sigma)}$;
		\item For every~$\alpha$, we have~$X_\alpha^{(0)}=X^{(0)}$ and~$X_\alpha$ is connected;
		\item For every~$\alpha$, the quotient~$G\backslash X_\alpha^{(n)}$ is compact.
	\end{enumerate}
	Then the group pair~$(G,\calP)$ is of type~$\sfFP_n$ if and only if for every~$i=1,\ldots,n-1$, the directed system of groups~$(H_i(X_\alpha;\IZ))_\alpha$ is essentially trivial.
	\begin{proof}
		Let~$C_*$ and~$C^\alpha_*$ denote the cellular $\IZ G$-chain complex of~$X$ and~$X_\alpha$, respectively.
		We have 
		\[
			C_0=C^\alpha_0=\IZ[G/\calP]\oplus \bigoplus_{v\in \Sigma_0}\IZ[G/G_v],
		\]
		where~$\Sigma_0$ is a set of $G$-orbit representatives of $0$-cells in~$X\smallsetminus G/\calP$.
		Consider the obvious augmentation map $\eta\colon C_0\onto \IZ$.
		Applying the snake lemma to the diagram
		\[\begin{tikzcd}
			0\ar{r}
			& \IZ[G/\calP]\ar{r}\ar{d}{\varepsilon}
			& C_0\ar{r}\ar{d}{\eta}
			& \bigoplus_{v\in \Sigma_0} \IZ[G/G_v]\ar{r}\ar{d}
			& 0
			\\
			0\ar{r}
			& \IZ\ar{r}{=}
			& \IZ\ar{r}
			& 0\ar{r}
			& 0
		\end{tikzcd}\]
		yields a short exact sequence of $\IZ G$-modules
		\begin{equation}
		\label{eqn:snake}
			0\to \Delta_{G/\calP}\to \ker(\eta)\to \bigoplus_{v\in \Sigma_0}\IZ[G/G_v]\to 0.
		\end{equation}
		Since~$\Sigma_0$ is finite and for every~$v\in \Sigma_0$, the group~$G_v$ is of type~$\sfFP_n$, the $\IZ G$-module~$\bigoplus_{v\in \Sigma_0}\IZ[G/G_v]$ is of type~$\sfFP_n$.
		Then it follows from the short exact sequence~\eqref{eqn:snake} that the $\IZ G$-module~$\Delta_{G/\calP}$ is of type~$\sfFP_{n-1}$ if and only if~$\ker(\eta)$ is of type~$\sfFP_{n-1}$~\cite[Proposition~1.4]{Bieri81}.
		
		Consider the truncated $\IZ G$-chain complex~$C_{*\ge 1}$ augmented over~$\ker(\eta)$ via the map~$\partial_1\colon C_1\onto \ker(\eta)$.
		For every~$\alpha$, the $\IZ G$-subcomplex~$C^\alpha_{*\ge 1}$ of~$C_{*\ge 1}$ is augmented by assumption~(iii).
		Note that for every~$i\ge 1$, we have
		\[
			\widetilde{H}_i(C^\alpha_{*\ge 1})\cong H_i(C^\alpha_*). 
		\]
		Then the claim follows by applying Proposition~\ref{prop:Brown_criterion_chain} for~$n-1$ to the $\IZ G$-chain complex~$C_{*\ge 1}$ augmented over~$\ker(\eta)$ and the filtration~$(C^\alpha_{*\ge 1})_\alpha$ of~$C_{*\ge 1}$.
	\end{proof}
\end{thm}

We record the special case of Theorem~\ref{thm:Brown_criterion_space} separately, when~$X$ itself has cocompact $n$-skeleton and we take the constant filtration.
This is a relative version of~\cite[Proposition~1.1]{Brown87} for group pairs.

\begin{cor}
\label{cor:combination}
	Let~$(G,\calP)$ be a group pair and let~$n\in \IN$ with~$n\ge 1$.
	Let~$(X,G/\calP)$ be a $G$-CW-pair.
	Suppose that the following hold:
	\begin{enumerate}[label=\enum]
		\item $X$ is $(n-1)$-acyclic;
		\item For every cell~$\sigma$ in~$X\smallsetminus G/\calP$, the stabiliser~$G_\sigma$ is of type~$\sfFP_{n-\dim(\sigma)}$;
		\item The quotient~$G\backslash X^{(n)}$ is compact.
	\end{enumerate}
	Then the group pair~$(G,\calP)$ is of type~$\sfFP_n$.
\end{cor}

\begin{rem}
\label{rem:FPnR}
	The homological finiteness properties can be defined over an arbitrary ring~$R$.
	A group~$G$ is \emph{of type~$\sfFP_n(R)$} if the trivial $RG$-module~$R$ is of type~$\sfFP_n$.
	A group pair~$(G,\calP)$ is \emph{of type~$\sfFP_n(R)$} if the $RG$-module~$\ker(R[G/\calP]\to R)$ is of type~$\sfFP_{n-1}$.
	Theorem~\ref{thm:Brown_criterion_space} holds more generally for type~$\sfFP_n(R)$ by considering homology with coefficients in~$R$.
\end{rem}

We conclude this section by relating the homological finiteness properties (over~$\IZ$) of group pairs to the geometric ones (Definition~\ref{defn:Fn_pairs}).
The proof is analogous to the corresponding result for groups, see~\cite[Theorem~VIII.7.1]{Brown82} and~\cite[Section~4]{Alonso91}.
We say that a $G$-CW-pair~$(X,A)$ is \emph{relatively free} if every cell in~$X\smallsetminus A$ has trivial stabiliser.

\begin{thm}
\label{thm:Eilenberg-Ganea}
	Let~$(G,\calP)$ be a group pair and let~$n\in \IN$ with~$n\ge 2$.
	Then~$(G,\calP)$ is of type~$\sfF_n$ if and only if~$(G,\calP)$ is relatively finitely presented and of type~$\sfFP_n$.
	\begin{proof}
		Suppose that~$(G,\calP)$ is of type~$\sfF_n$.
		Let~$(X,G/\calP)$ be a relatively free $G$-CW-pair, where~$X$ is contractible and has cocompact $n$-skeleton.
		Then~$(G,\calP)$ is of type~$\sfFP_n$ by Corollary~\ref{cor:combination}. 
		Moreover, a relative finite presentation for~$(G,\calP)$ can be obtained as follows.
		A ``blow-up" construction~\cite{Haefliger} produces from~$(X,G/\calP)$ a free $G$-CW-pair $(EG,\coprod_{P\in \calP}G\times_P EP)$ with only finitely many $G$-orbits of cells of dimension~$\le n$ in $EG\smallsetminus \coprod_{P\in \calP} G\times_P EP$.
		By dividing out the $G$-action, we obtain a CW-pair $(BG,\coprod_{P\in \calP} BP)$ with only finitely many cells of dimension~$\le n$ in~$BG\smallsetminus \coprod_{P\in \calP} BP$.
		Then the $2$-skeleton of~$BG$ yields a finite relative presentation for~$(G,\calP)$.
		
		Suppose that~$(G,\calP)$ is relatively finitely presented and of type~$\sfFP_n$.
		Let~$\spann{S,\calP\mid R}$ be a finite relative presentation for~$(G,\calP)$.
		It suffices to construct a relatively free $G$-CW-pair~$(X,G/\calP)$ such that~$X$ is $n$-dimensional, $(n-1)$-acyclic, simply-connected, and cocompact. 
		We proceed by induction over the skeleta.
		Let~$X^{(1)}$ be the coned-off Cayley graph~$\hat \Gamma(G,\calP,S)$ (Definition~\ref{defn:coned-off Cayley}).
		Since the relative presentation~$\spann{S,\calP\mid R}$ is finite, by Proposition~\ref{prop:coned-off Cayley_pi1} we can attach finitely many free $G$-orbits of $2$-cells to~$X^{(1)}$ to obtain a simply-connected $2$-complex~$X^{(2)}$.
		Suppose that we have constructed a relatively free $G$-CW-pair~$(X^{(n-1)},G/\calP)$, where~$X^{(n-1)}$ is $(n-1)$-dimensional, $(n-2)$-acyclic, simply-connected, and cocompact.
		We denote by~$C_*(X^{(n-1)})$ the cellular $\IZ G$-chain complex of~$X^{(n-1)}$ which is augmented over~$\IZ$.
		Since~$X^{(n-1)}$ is $(n-2)$-acyclic, the following sequence of~$\IZ G$-modules is exact
		\[
			0\to H_{n-1}(X^{(n-1)};\IZ)\to C_{n-1}(X^{(n-1)})\to \cdots\to C_0(X^{(n-1)})\to \IZ\to 0.
		\]
		Since the $G$-CW-pair~$(X^{(n-1)},G/\calP)$ is relatively free and~$X^{(n-1)}$ is cocompact, for all~$i=1,\ldots,n-1$, the $\IZ G$-module~$C_i(X^{(n-1)})$ is finitely generated free, and~$C_0(X^{(n-1)})$ is of the form~$L\oplus \IZ[G/\calP]$, where~$L$ is a finitely generated free $\IZ G$-module.
		Since the group pair~$(G,\calP)$ is of type~$\sfFP_n$, it follows from Lemma~\ref{lem:0-unfree_resolutions} that the $\IZ G$-module~$H_{n-1}(X^{(n-1)};\IZ)$ is finitely generated.
		Let~$(z_j)_j$ be a finite $\IZ G$-generating set of~$H_{n-1}(X^{(n-1)};\IZ)$.
		Since~$X^{(n-1)}$ is simply-connected and $(n-2)$-acyclic, the Hurewicz theorem yields maps $(f_j\colon S^{n-1}\to X^{(n-1)})_j$ representing the homology classes~$(z_j)_j$.
		Finally, we construct the desired space~$X=X^{(n)}$ by attaching free $G$-orbits of $n$-cells to~$X^{(n-1)}$ along the maps~$(f_j)_j$ and their $G$-translates.
	\end{proof}
\end{thm}

\section{Quasi-isometry invariance}

We prove that the homological finiteness properties and relative finite presentability of group pairs are inherited by quasi-retracts and hence invariant under strong quasi-isometry.

\subsection{Homological finiteness properties of group pairs}

The proof that homological finiteness properties of groups are inherited by quasi-retracts~\cite{Alonso94} goes roughly as follows: 
By Brown's criterion, a group is of type~$\sfFP_n$ if and only if the filtration of Rips complexes is essentially $(n-1)$-acyclic, and the latter is inherited by quasi-retracts.
We adapt these arguments to the relative setting of group pairs.

\begin{defn}
\label{defn:unicone Rips}
	Let~$(G,\calP)$ be a group pair and let~$S$ be a finite generating set of~$G$.
	We denote by~$d_S$ the word metric on~$G$.
	We consider the set~$G\sqcup G/\calP$ and call elements in~$G/\calP$ \emph{cone vertices}.
	Let~$\alpha\in \IN$.
	
	A finite subset~$U$ of~$G\sqcup G/\calP$ is a \emph{unicone subset of diameter~$\le \alpha$} if it satisfies the following conditions:
	\begin{itemize}
		\item The diameter of~$U\cap G$ in~$(G,d_S)$ is bounded by~$\alpha$;
		\item The set~$U$ contains at most one cone vertex~$A$. 
		For every~$g\in U\cap G$, the distance in~$(G,d_S)$ between~$g$ and~$A$ is bounded by~$\alpha$, where the coset~$A$ is viewed as a subset of~$G$. 
		Explicitly, for every~$g\in U\cap G$, there exists~$a\in A$ with~$d_S(g,a)\le \alpha$.
	\end{itemize}
	The \emph{unicone $\alpha$-Rips complex~$\hat R_\alpha(G,\calP,S)$} is the simplicial complex whose simplices are all the unicone subsets of~$G\sqcup G/\calP$ of diameter~$\le \alpha$.
	For~$\beta\in \IN$ with~$\beta\ge \alpha$, $\hat R_\alpha(G,\calP,S)$ is a subcomplex of~$\hat R_\beta(G,\calP,S)$ and we define
	\[
		\hat R_\infty(G,\calP,S)\coloneqq \varinjlim_{\alpha\in \IN} \hat R_\alpha(G,\calP,S).
	\]
\end{defn}

The unicone Rips complex is natural with respect to Lipschitz maps of group pairs in the following sense.

\begin{lem}
\label{lem:unicone Rips_functorial}
	Let~$(G,\calP)$ and~$(H,\calQ)$ be group pairs and let~$S$ and~$T$ be finite generating sets of~$G$ and~$H$, respectively.
	Let~$f=(f_1,f_2)\colon (G,\calP)\to (H,\calQ)$ be an $(L,C,M)$-Lipschitz map of pairs.
	Then the following hold:
	\begin{enumerate}[label=\enum]
		\item For all~$\alpha,\alpha'\in \IN$ with~$\alpha'\ge L\alpha+C+M$, there is a simplicial map 
		\[
			f_*\colon \hat R_\alpha(G,\calP,S)\to \hat R_{\alpha'}(H,\calQ,T)
		\]
		given on vertices by~$f_1\sqcup f_2\colon G\sqcup G/\calP\to H\sqcup H/\calQ$;
		\item Let~$r\colon (H,Q)\to (G,\calP)$ be an $(L,C,M)$-Lipschitz map of pairs such that $(f,r)$ is a quasi-retraction.
		For all~$\alpha,\alpha',\beta,\beta'\in \IN$ with~$\alpha'\ge L\alpha+C+M$, $\beta'\ge \alpha'$, and~$\beta\ge L\beta'+2C+M$, the following diagram commutes up to simplicial homotopy
		\[\begin{tikzcd}
			\hat R_\alpha(G,\calP,S)\ar{r}{f_*}\ar[hook]{d}
			& \hat R_{\alpha'}(H,\calQ,T)\ar[hook]{d}
			\\
			\hat R_{\beta}(G,\calP,S)
			& \hat R_{\beta'}(H,\calQ,T)\ar{l}[swap]{r_*}	
		\end{tikzcd}\]
		In particular, for every~$i\in \IN$, if the directed system of homology groups $(H_i(\hat R_{\alpha'}(H,\calQ,T);\IZ))_{\alpha'\in \IN}$ is essentially trivial, then the directed system of homology groups~$(H_i(\hat R_\alpha(G,\calP,S);\IZ))_{\alpha\in \IN}$ is essentially trivial.
	\end{enumerate}
	\begin{proof}
		We denote~$\hat G\coloneqq G\sqcup G/\calP$, $\hat H\coloneqq H\sqcup H/\calQ$, and~$\hat f\coloneqq f_1\sqcup f_2\colon \hat G\to \hat H$.
		
		(i) Let~$U$ be a unicone subset of diameter~$\le \alpha$ in~$\hat G$.
		We show that the subset~$\hat f(U)=f_1(U\cap G)\sqcup f_2(U\cap G/\calP)$ of~$\hat H$ is a unicone subset of diameter~$\le L\alpha+C+M$ in~$\hat H$.
		Indeed, since~$\diam(U\cap G,d_S)\le \alpha$ and~$f_1$ is $(L,C)$-Lipschitz, we have~$\diam(f_1(U\cap G),d_T)\le L\alpha+C$.
		If~$U$ contains a unique cone vertex~$A$, then~$f_2(A)$ is the unique cone vertex of~$\hat f(U)$.
		For every~$g\in U\cap G$, there exists~$a\in A$ with~$d_S(g,a)\le \alpha$.
		Since~$f$ is an $(L,C,M)$-Lipschitz map of pairs, there exists~$b\in f_2(A)$ such that~$d_T(f_1(a),b)\le M$.
		Together, we have
		\[
			d_T\bigl(f_1(g),b\bigr) 
			\le d_T\bigl(f_1(g),f_1(a)\bigr) + d_T\bigl(f_1(a),b\bigr)
			\le L\alpha +C +M.
		\]
		In other words, the distance in~$(H,d_T)$ between any element of~$f_1(U\cap G)$ and the subset~$f_2(A)$ is bounded by~$L\alpha +C+M$.
		
		(ii) Let~$U$ be a unicone subset of diameter~$\le \alpha$ in~$\hat G$.
		By part~(i), then~$\hat r\circ \hat f(U)$ is a unicone subset of diameter~$\le L(L\alpha+C+M)+C+M$ in~$\hat G$.
		Since~$(f,r)$ is an $(L,C,M)$-quasi-retraction of pairs, the subset~$\hat r\circ \hat f(U)$ of~$\hat G$ contains the same cone vertex (if any) as~$U$, and the union~$\hat r\circ \hat f(U)\cup U$ is a unicone subset of diameter~$\le L(L\alpha+C+M)+C+M+C$ in~$\hat G$.
		Since~$\beta$ is larger than this diameter, moving every vertex~$v$ of~$\hat R_\alpha(G,\calP,S)$ linearly to~$\hat r\circ \hat f(v)$ inside~$\hat R_\beta(G,\calP,S)$ extends to the desired simplicial homotopy $\hat R_\alpha(G,\calP,S)\times [0,1]\to \hat R_\beta(G,\calP,S)$.
	\end{proof}
\end{lem}

If~$S$ and~$S'$ are finite generating sets of~$G$, then the identity on~$G$ induces a strong quasi-isometry of pairs~$(G,\calP,S)\to (G,\calP,S')$.
Lemma~\ref{lem:unicone Rips_functorial} shows, in particular, that~$(H_i(\hat R_\alpha(G,\calP,S)))_{\alpha\in \IN}$ being essentially trivial is independent of the finite generating set~$S$ of~$G$.

\begin{rem}
	We kept Definition~\ref{defn:unicone Rips} and Lemma~\ref{lem:unicone Rips_functorial} short and self-contained, but some readers may find the following equivalent description of the unicone Rips complex more conceptual.
	Let~$(G,\calP)$ be a group pair and let~$S$ be a finite generating set of~$G$.
	By considering a coset in~$G/\calP$ as a subset of~$G$, the word metric~$d_S$ on~$G$ extends to a ``distance function"~$\hat d_S$ on~$G\sqcup G/\calP$.
	In general, $\hat d_S$ is not a metric because it is not faithful nor satisfies the triangle inequality.
	However, the usual definition of the $\alpha$-Rips complex still makes sense, as the simplicial complex whose simplices are all the finite subset of~$(G\sqcup G/\calP,\hat d_S)$ of diameter~$\le \alpha$.
	This construction was considered in~\cite{MPP}.
	Then the unicone $\alpha$-Rips complex~$\hat R_\alpha(G,\calP,S)$ is the subcomplex of the $\alpha$-Rips complex on~$(G\sqcup G/\calP,\hat d_S)$ consisting of all subsets that contain at most one element of~$G/\calP$.
	To obtain Lemma~\ref{lem:unicone Rips_functorial} in this description, one shows that a Lipschitz map of group pairs~$(G,\calP)\to (H,\calQ)$ induces a Lipschitz map~$(G\sqcup G/\calP,\hat d_S)\to (H\sqcup H/\calQ,\hat d_T)$ and then uses the naturality of the (unicone) Rips complex.
\end{rem}

The isometric $G$-action on the Cayley graph~$\Gamma(G,S)$ extends to a simplicial $G$-action on the unicone Rips complex~$\hat R_\alpha(G,\calP,S)$ for every~$\alpha\in \IN$ and hence on~$\hat R_\infty(G,\calP,S)$.
In the following, we do not distinguish between a simplicial complex and its geometric realisation.

\begin{lem}
\label{lem:unicone Rips_cocompact}
	Let~$(G,\calP)$ be a group pair and let~$S$ be a finite generating set of~$G$.
	The following hold:
	\begin{enumerate}[label=\enum]
		\item $\hat R_\infty(G,\calP,S)$ is (non-equivariantly) contractible;
		\item For every~$\alpha\in \IN$, the barycentric subdivision~$\hat R_\alpha(G,\calP,S)'$ of~$\hat R_\alpha(G,\calP,S)$ is a cocompact $G$-CW-complex.
		Moreover, every cell in~$\hat R_\alpha(G,\calP,S)'\smallsetminus G/\calP$ has finite stabiliser.
	\end{enumerate}
	\begin{proof}
		(i) We have a pushout of the form
		\[\begin{tikzcd}
			\coprod_{A\in G/\calP} R_\infty(G,S)\ar{r}{\coprod \id}\ar[hook]{d}
			& R_\infty(G,S)\ar{d}
			\\
			\coprod_{A\in G/\calP} \Cone_A(R_\infty(G,S))\ar{r}
			& \hat R_\infty(G,\calP,S)
		\end{tikzcd}\]
		where~$R_\infty(G,S)$ is the $\infty$-Rips complex on~$(G,d_S)$ and~$\Cone_A$ is the cone with cone point~$A$.
		Since~$R_\infty(G,S)$ is contractible, also~$\hat R_\infty(G,\calP,S)$ is contractible.
		
		(ii) We denote~$\hat G\coloneqq G\sqcup G/\calP$ and~$\hat R_\alpha\coloneqq \hat R_\alpha(G,\calP,S)$.
		We show that there are only finitely many $G$-orbits of unicone subsets of diameter~$\le \alpha$ in~$\hat G$.
		Indeed, since~$(G,d_S)$ is locally finite, there are only finitely many $G$-orbits of unicorn subsets of diameter~$\le \alpha$ in~$\hat G$ containing no cone vertex.
		Let~$U$ be a unicone subset of diameter~$\le \alpha$ in~$\hat G$ containing exactly one cone vertex.
		Since the collection~$\calP$ is finite, there are only finitely many $G$-orbits of such~$U$ of cardinality one.
		Assume that~$U$ is of the form~$\{A,v_1,\ldots,v_l\}$ with~$A\in G/\calP$, $v_1,\ldots,v_l\in G$, and~$l\ge 1$.
		Up to $G$-translation we may assume that~$v_1$ is the neutral element~$e\in G$.
		Then~$A$ is one of only finitely many cosets in~$G/\calP$ that have distance~$\le \alpha$ to~$e$ in~$(G,d_S)$.
		The elements~$v_2,\ldots,v_l$ lie in the ball of radius~$\alpha$ centred at~$e$ in~$(G,d_S)$ which is finite, again by local finiteness.
		It follows that the simplicial $G$-action on~$\hat R_\alpha$ is cocompact.
		The barycentric subdivision of any simplicial complex with a simplicial $G$-action is a $G$-CW-complex and cocompactness is preserved.
		
		We show that every cell in~$\hat R_\alpha'\smallsetminus G/\calP$ has finite stabiliser.
		A vertex~$\sigma$ of the barycentric subdivision~$\hat R_\alpha'$ corresponds to a simplex~$\{v_0,\ldots,v_l\}$ of~$\hat R_\alpha$.
		The stabiliser~$G_\sigma$ is the setwise stabiliser of~$\{v_0,\ldots,v_l\}$ which is a finite extension of the pointwise stabiliser~$\bigcap_{i=0}^l G_{v_i}$, where~$G_{v_i}$ is the stabiliser of~$v_i\in \hat G$.
		If~$v_i\in G$, then~$G_{v_i}$ is trivial, and if~$v_i\in G/\calP$, then~$G_{v_i}$ is a conjugate of a group in~$\calP$.
		In particular, if~$\sigma$ contains an element of~$G$, then the stabiliser~$G_\sigma$ is finite.
		An edge~$\tau$ of~$\hat R_\alpha'$ corresponds to a strict inclusion of simplices~$\{w_0,\ldots,w_k\}\subset \{v_0,\ldots,v_l\}$ of~$\hat R_\alpha$ for~$0\le k<l$.
		The stabiliser~$G_\tau$ is the subgroup of the setwise stabiliser of~$\sigma\coloneqq \{v_0,\ldots,v_l\}$ that setwise fixes~$\{w_0,\ldots,w_k\}$.
		In particular, the stabiliser~$G_\tau$ is a subgroup of~$G_\sigma$ which is finite because~$\sigma$ contains at most one cone vertex and hence contains at least one element of~$G$.
		Stabilisers of cells of dimension~$\ge 2$ are subgroups of edge stabilisers and therefore also finite.
	\end{proof}
\end{lem}

The homological finiteness properties of group pairs can be characterised in terms of unicone Rips complexes.

\begin{prop}
\label{prop:FPn_unicone Rips}
	Let~$(G,\calP)$ be a group pair and let~$n\in \IN$ with~$n\ge 2$.
	Then~$(G,\calP)$ is of type~$\sfFP_n$ if and only if for every~$i=1,\ldots,n-1$, the directed system of homology groups~$(H_i(\hat R_\alpha(G,\calP);\IZ))_{\alpha\in \IN}$ is essentially trivial.
	\begin{proof}
		This follows by applying Theorem~\ref{thm:Brown_criterion_space} to the $G$-CW-pair~$(\hat R_\infty(G,\calP)',G/\calP)$ and the filtration $(\hat R_\alpha(G,\calP)')_{\alpha\ge 1}$ of~$\hat R_\infty(G,\calP)'$.
		The conditions of Theorem~\ref{thm:Brown_criterion_space} are satisfied by Lemma~\ref{lem:unicone Rips_cocompact}.
	\end{proof}
\end{prop}

We deduce the inheritance of the homological finiteness properties of group pairs under quasi-retracts.

\begin{thm}
\label{thm:FPn_QI}
	Let~$(G,\calP)$ and~$(H,\calQ)$ be group pairs and let~$n\in \IN$ with~$n\ge 2$.
	If~$(G,\calP)$ is a quasi-retract of~$(H,\calQ)$ and~$(H,\calQ)$ is of type~$\sfFP_n$, then~$(G,\calP)$ is of type~$\sfFP_n$.
	\begin{proof}
		By Proposition~\ref{prop:FPn_unicone Rips}, since~$(H,\calQ)$ is of type~$\sfFP_n$, the directed system of groups $(H_i(\hat R_{\alpha'}(H,\calQ);\IZ))_{\alpha'\in \IN}$ is essentially trivial for every~$i=1,\ldots,n-1$.
		Since~$(G,\calP)$ is a quasi-retract of~$(H,\calQ)$, Lemma~\ref{lem:unicone Rips_functorial} yields that the directed system of groups $(H_i(\hat R_\alpha(G,\calP);\IZ))_{\alpha\in \IN}$ is essentially trivial for every~$i=1,\ldots,n-1$.
		Then the group pair~$(G,\calP)$ is of type~$\sfFP_n$, again by Proposition~\ref{prop:FPn_unicone Rips}.
	\end{proof}
\end{thm}

\begin{rem}
	Theorem~\ref{thm:FPn_QI} holds more generally for type~$\sfFP_n(R)$ over a ring~$R$ (Remark~\ref{rem:FPnR}) and homology with coefficients in~$R$ by the same proof.
\end{rem}

\subsection{Relative finite presentability}
The classical proof that finite presentability of groups is inherited by quasi-retracts goes roughly as follows:
A group is finitely presented if and only if its Cayley graph is coarsely simply-connected, and the latter is inherited by quasi-retracts.
We adapt these arguments to the relative setting of group pairs.

Let~$\Lambda$ be a simplicial graph and let~$l\in \IN$.
An \emph{(edge) path of length~$l-1$} in~$\Lambda$ is a sequence~$v_1,v_2,\ldots,v_l$ of vertices in~$\Lambda$ such that~$\{v_i,v_{i+1}\}$ is an edge in~$\Lambda$ for every~$i=1,\ldots,l-1$.
An \emph{(edge) loop of length~$l$} in~$\Lambda$ is an edge path~$v_1,v_2,\ldots,v_l$ in~$\Lambda$ such that~$\{v_l,v_1\}$ is an edge in~$\Lambda$.
In the following, we do not distinguish between an edge loop and its geometric realisation.

\begin{defn}
\label{defn:unicone loop}
	Let~$(G,\calP)$ be a group pair and let~$S$ be a finite generating set of~$G$.
	We denote by~$\hat \Gamma$ the coned-off Cayley graph~$\hat \Gamma(G,\calP,S)$ (Definition~\ref{defn:coned-off Cayley}).
	The vertices of~$\hat \Gamma$ in~$G/\calP$ are called \emph{cone vertices}.
	An edge loop in~$\hat \Gamma$ is a \emph{unicone loop} if it contains at most one cone vertex.

	For~$l\in \IN$, we denote by~$\hat \Gamma_l$ the $2$-complex with $1$-skeleton~$\hat \Gamma$ and a $2$-cell attached to each unicone loop of length~$<l$ in~$\hat \Gamma$.
	The graph~$\hat \Gamma$ is \emph{coarsely unicone simply-connected} if there exists~$l\in \IN$ such that~$\hat \Gamma_l$ is simply-connected.
\end{defn}

Being coarsely unicone simply-connected is inherited by quasi-retracts.

\begin{lem}
\label{lem:unicone loop_functorial}
	Let~$(G,\calP)$ and~$(H,\calQ)$ be group pairs and let~$S$ and~$T$ be finite generating sets of~$G$ and~$H$, respectively.
	If~$(G,\calP)$ is a quasi-retract of~$(H,\calQ)$ and the graph~$\hat \Gamma(H,\calQ,T)$ is coarsely unicone simply-connected, then the graph~$\hat \Gamma(G,\calP,S)$ is coarsely unicone simply-connected.
	\begin{proof}
		We denote~$\hat \Gamma_G\coloneqq \hat \Gamma(G,\calP,S)$ and~$\hat \Gamma_H\coloneqq \hat \Gamma(H,\calQ,T)$.
		Let~$(f,r)$ be a quasi-retraction of pairs, where~$f\colon (G,\calP)\to (H,\calQ)$ and~$r\colon (H,\calQ)\to (G,\calP)$ are Lipschitz maps of pairs.
		Lemma~\ref{lem:coned-off Cayley_functorial} yields a cellular map~$\hat f_*\colon \hat \Gamma_G\to \hat \Gamma_H$ and a constant~$\hat L\in \IR$, which depends only on the quasi-isometry constants.		
		Then, for every loop~$\gamma$ of length~$l$ in~$\hat \Gamma_G$, we get an associated loop~$\hat f_*\gamma$ of length~$\le \hat{L}\cdot l$ in~$\hat \Gamma_H$ with the same number of cone vertices as~$\gamma$.
		In particular, $\hat f_*\gamma$ is a unicone loop if and only if~$\gamma$ is a unicone loop.
		Similarly, Lemma~\ref{lem:coned-off Cayley_functorial} yields a cellular map~$\hat r_*\colon \hat \Gamma_H\to \hat \Gamma_G$ such that for every (unicone) loop~$\delta$ of length~$\le l$ in~$\hat \Gamma_H$, we get an associated (unicone) loop~$\hat r_*\delta$ of length~$\le \hat L\cdot l$ in~$\hat \Gamma_G$.
		Then~$\hat r_*(\hat f_*\gamma)$ is a loop in~$\hat \Gamma_G$ with the same (if any) cone vertices as~$\gamma$.
		Since~$(f,r)$ is a quasi-retraction, there exists~$l_1\in \IN$, depending only on the quasi-isometry constants, such that every loop~$\gamma$ in~$\hat \Gamma_G$ is homotopic to~$\hat r_*(\hat f_*\gamma)$ in~$\hat \Gamma_{G,l_1}$.
		Indeed, the loops~$\gamma$ and~$\hat r_*(\hat f_*\gamma)$ can be connected by an ``annulus" consisting of $2$-cells attached along unicone loops of uniformly bounded length.
		
		Let~$l_2\in \IN$ be such that~$\hat \Gamma_{H,l_2}$ is simply-connected.
		Set~$l_3\coloneqq \hat L\cdot l_2$.
		Then the map~$\hat r_*\colon \hat \Gamma_H\to \hat \Gamma_G$ extends to a map of 2-complexes~$\hat \Gamma_{H,l_2}\to \hat \Gamma_{G,l_3}$.
		Hence, for every loop~$\delta$ in~$\hat \Gamma_H$, by pushing forward the nullhomotopy of~$\delta$ in~$\hat \Gamma_{H,l_2}$, the loop~$\hat r_*\delta$ is nullhomotopic in~$\hat \Gamma_{G,l_3}$.
		Together, for~$l_4\coloneqq \max\{l_1,l_3\}$, every loop~$\gamma$ in~$\hat \Gamma_G$ is nullhomotopic in~$\hat \Gamma_{G,l_4}$.
	\end{proof}
\end{lem}

If~$S$ and~$S'$ are finite generating sets of~$G$, then the identity on~$G$ induces a strong quasi-isometry of pairs~$(G,\calP,S)\to (G,\calP,S')$.
Lemma~\ref{lem:unicone loop_functorial} shows, in particular, that~$\hat \Gamma(G,\calP,S)$ being coarsely unicone simply-connected is independent of the finite generating set~$S$ of~$G$.

\begin{lem}
\label{lem:unicone loop_cocompact}
	Let~$(G,\calP)$ be a group pair and let~$S$ be a finite generating set of~$G$.
	The following hold:
	\begin{enumerate}[label=\enum]
		\item For every loop~$\gamma$ in~$\hat \Gamma (G,\calP,S)$, there exists~$l\in \IN$ such that~$\gamma$ is nullhomotopic in the $2$-complex~$\hat \Gamma(G,\calP,S)_l$;
		\item For every~$l\in \IN$, there are only finitely many $G$-orbits of unicone loops of length~$l$ in~$\hat \Gamma(G,\calP,S)$.
	\end{enumerate}
	\begin{proof}
		We denote~$\hat \Gamma\coloneqq \hat \Gamma(G,\calP,S)$.

		 (i) We proceed by induction on the number of cone vertices of~$\gamma$.
		 If~$\gamma$ is a unicone loop, the claim is true.
		 Otherwise, choose a cone vertex~$A$ of~$\gamma$ and let~$v$ and~$w$ be the vertices adjacent to~$A$ along~$\gamma$.
		 Since~$v$ and~$w$ lie in~$G$, there exists a path~$\zeta$ from~$w$ to~$v$ consisting of vertices in~$G$.
		 Then the loop~$A,\zeta$ is a unicone loop of a certain length, say~$<l_1$.
		 It follows that~$\gamma$ is homotopic in~$\hat \Gamma_{l_1}$ to a loop~$\gamma'$ with one cone vertex less than~$\gamma$.
		 By induction, the loop~$\gamma'$ is nullhomotopic in~$\hat \Gamma_{l_2}$ for some~$l_2\in \IN$.
		 Hence, for~$l\coloneqq \max\{l_1,l_2\}$, the loop~$\gamma$ is nullhomotopic in~$\hat \Gamma_l$.
		 
		 (ii) Since the Cayley graph~$\Gamma(G,S)$ is locally finite, there are only finitely many $G$-orbits of loops of length~$l$ with no cone vertex.
		 Let~$\gamma$ be a loop of length~$l$ in~$\hat \Gamma$ with exactly one cone vertex~$A$.
		 Up to $G$-translation, we may assume that one of the vertices adjacent to~$A$ along~$\gamma$ is the neutral element~$e\in G$.
		 Hence~$A$ is one of the finitely many cosets~$(eP)_{P\in \calP}$.
		 The other~$l-2$ vertices of~$\gamma$ lie in the ball of radius~$l-2$ centred at~$e$ in~$\Gamma(G,S)$ which is finite, again by local finiteness.
	\end{proof}
\end{lem}

The relative finite presentability of group pairs can be characterised in terms of unicone loops.

\begin{prop}
\label{prop:fin pres_unicone loop}
	Let~$(G,\calP)$ be a group pair.
	Then~$(G,\calP)$ is relatively finitely presented if and only if the graph~$\hat \Gamma(G,\calP)$ is coarsely unicone simply-connected.
	\begin{proof}		
		Suppose that~$(G,\calP)$ admits a finite relative presentation~$\spann{S,\calP\mid R}$.
		We may assume that~$S$ is a finite (non-relative) generating set of~$G$.
		By Proposition~\ref{prop:coned-off Cayley_pi1}, we can obtain a simply-connected $2$-complex by attaching finitely many $G$-orbits of $2$-cells to~$\hat \Gamma(G,\calP,S)$.
		By Lemma~\ref{lem:unicone loop_cocompact}~(i), we may assume that these finitely many $G$-orbits of $2$-cells are attached along unicone loops in~$\hat \Gamma(G,\calP,S)$.
		Hence the graph~$\hat \Gamma(G,\calP,S)$ is coarsely unicone simply-connected.
		
		Suppose that~$\hat \Gamma(G,\calP,S)$ is coarsely unicone simply-connected for some finite generating set~$S$ of~$G$.
		Let~$l\in \IN$ be such that~$\hat \Gamma(G,\calP,S)_l$ is simply-connected.
		By Lemma~\ref{lem:unicone loop_cocompact}~(ii), there are only finitely many $G$-orbits of unicone loops of length~$<l$ in~$\hat \Gamma(G,\calP,S)$.
		Then the $2$-complex~$\hat \Gamma(G,\calP,S)_l$ is obtained from~$\hat \Gamma(G,\calP,S)$ by attaching finitely many $G$-orbits of $2$-cells (along unicone loops).
		Hence there exists a finite relative presentation~$\spann{S,\calP\mid R}$ of~$(G,\calP)$ by Proposition~\ref{prop:coned-off Cayley_pi1}.
	\end{proof}
\end{prop}

We deduce the inheritance of relative finite presentability of group pairs under quasi-retracts.

\begin{thm}
\label{thm:fin pres_QI}
	Let~$(G,\calP)$ and~$(H,\calQ)$ be group pairs.
	If~$(G,\calP)$ is a quasi-retract of~$(H,\calQ)$ and~$(H,\calQ)$ is relatively finitely presented, then~$(G,\calP)$ is relatively finitely presented.
	\begin{proof}
		By Proposition~\ref{prop:fin pres_unicone loop}, since~$(H,\calQ)$ is relatively finitely presented, the graph $\hat \Gamma(H,\calQ)$ is coarsely unicone simply-connected.
		Since~$(G,\calP)$ is a quasi-retract of~$(H,\calQ)$, Lemma~\ref{lem:unicone loop_functorial} yields that the graph~$\hat \Gamma(G,\calP)$ is coarsely unicone simply-connected.
		Then the group pair~$(G,\calP)$ is relatively finitely presented, again by Proposition~\ref{prop:fin pres_unicone loop}.
	\end{proof}
\end{thm}

\begin{cor}
\label{cor:Fn_QI}
	Let~$(G,\calP)$ and~$(H,\calQ)$ be group pairs and let~$n\in \IN$ with~$n\ge 2$.
	If~$(G,\calP)$ is a quasi-retract of~$(H,\calQ)$ and~$(H,\calQ)$ is of type~$\sfF_n$, then~$(G,\calP)$ is of type~$\sfF_n$.
	\begin{proof}
		By Theorem~\ref{thm:Eilenberg-Ganea}, for~$n\ge 2$, a group pair is of type~$\sfF_n$ if and only if it is relatively finitely presented and of type~$\sfFP_n$.
		Then the claim follows from Theorem~\ref{thm:fin pres_QI} and Theorem~\ref{thm:FPn_QI}.
	\end{proof}
\end{cor}

Corollary~\ref{cor:Fn_QI} and Theorem~\ref{thm:FPn_QI} together prove Theorem~\ref{thm:main_intro}.

\subsection{Bredon finiteness properties for families}

We relate the finiteness properties of group pairs to another well-studied notion of finiteness properties for groups relative to a family of subgroups, called Bredon finiteness properties.
For background on Bredon cohomology of groups, we refer to~\cite{Fluch}.

Let~$G$ be a group.
A \emph{family of subgroups of~$G$} is a non-empty set~$\calF$ of subgroups of~$G$ that is closed under taking subgroups and under conjugation by elements of~$G$.
Typical examples are the families consisting of the trivial subgroup, all finite subgroups, and all virtually cyclic subgroups, respectively.

Let~$\calF$ be a family of subgroups of~$G$.
A $G$-CW-complex~$Y$ is a model for the \emph{classifying space~$\EFG$} if for every subgroup~$K$ of~$G$, the fixed-point set~$Y^K$ is contractible if~$K\in \calF$, and empty otherwise.
Equivalently, $Y$ is a terminal object in the $G$-homotopy category of $G$-CW-complexes whose isotropy groups lie in~$\calF$.

The \emph{orbit category~$\OFG$} has as objects $G$-sets of the form~$G/K$ with~$K\in \calF$ and as morphisms $G$-maps.
The category of \emph{$\OFG$-modules} is the category of contravariant functors from~$\OFG$ to the category of abelian groups.
The trivial $\OFG$-module~$\underline{\IZ}$ sends every object to~$\IZ$ and every morphism to the identity on~$\IZ$.
The cellular $\OFG$-chain complex~$\underline{C}_*(Y)$ of a $G$-CW-complex~$Y$ evaluated at~$G/K$ is the cellular $\IZ$-chain complex of the fixed-point set~$Y^K$.

\begin{defn}
\label{defn:Bredon}
	Let~$G$ be a group, let~$\calF$ be a family of subgroups of~$G$, and let~$n\in \IN$.
		The group~$G$ is \emph{of type~$\calF$-$\sfF_n$} if there exists a $G$-CW-model for~$\EFG$ with cocompact $n$-skeleton.
		The group~$G$ is \emph{of type~$\calF$-$\sfFP_n$} if the trivial $\OFG$-module~$\underline{\IZ}$ is of type~$\sfFP_n$.
\end{defn}

The finiteness properties of a group are equivalent to the Bredon finiteness properties relative to the family containing only the trivial subgroup.

To a collection~$\calP$ of subgroups of~$G$, we associate the smallest family~$\calFP$ of subgroups of~$G$ containing all groups in~$\calP$.
Explicitly, the family~$\calFP$ consists of all conjugates of groups in~$\calP$ and all their subgroups.
The relationship between the cohomology of the group pair~$(G,\calP)$ and the Bredon cohomology of~$G$ relative to~$\calFP$ was investigated in~\cite{ANCMSS}.
When the collection~$\calP$ is malnormal (Definition~\ref{defn:reduced}), we show that the finiteness properties of the group pair and the Bredon finiteness properties are equivalent (Proposition~\ref{prop:Bredon_malnormal}). 

We begin with some general preparations.
For a $G$-CW-complex~$Y$, the \emph{singular $G$-subcomplex~$Y^\sing$} is the $G$-subcomplex of~$Y$ consisting of all cells with non-trivial stabiliser.
A $G$-homotopy equivalence between $G$-CW-complexes restricts to a $G$-homotopy equivalence between singular $G$-subcomplexes.
Algebraically, a free $\OFG$-module~$\underline{M}$ is of the form
\[
	\underline{M}=\bigoplus_{j\in J}\IZ[\Map_G(-,G/K_j)]
\]
for some set~$J$ and subgroups~$K_j\in \calF$.
The \emph{singular $\OFG$-submodule~$\underline{M}^\sing$} is the $\OFG$-submodule of~$\underline{M}$ given by
\[
	\underline{M}^\sing\coloneqq \bigoplus_{j\in J,K_j\neq 1} \IZ[\Map_G(-,G/K_j)].
\]
For a free $\OFG$-chain complex~$\underline{M}_*$, the \emph{singular $\OFG$-subcomplex~$\underline{M}_*^\sing$} is defined degree-wise by singular $\OFG$-submodules.
An $\OFG$-chain homotopy equivalence between free $\OFG$-chain complexes restricts to an $\OFG$-chain homotopy equivalence between singular $\OFG$-subcomplexes.

\begin{prop}
\label{prop:Bredon_malnormal}
	Let~$(G,\calP)$ be a group pair, where the collection~$\calP$ is malnormal and consists of distinct subgroups.
	Let~$\calFP$ denote the smallest family of subgroups of~$G$ containing~$\calP$.
	Let~$n\in \IN$ with~$n\ge 1$.
	Then the following hold:
	\begin{enumerate}[label=\enum]
		\item $(G,\calP)$ is of type~$\sfF_n$ if and only if~$G$ is of type~$\calFP$-$\sfF_n$;
		\item $(G,\calP)$ is of type~$\sfFP_n$ if and only if~$G$ is of type~$\calFP$-$\sfFP_n$.
	\end{enumerate}
	\begin{proof}
		Since the collection~$\calP$ is malnormal and consists of distinct subgroups, there exists a model~$Z$ for~$\EFPG$ that is a $G$-pushout of the following form
		\begin{equation}
		\label{eqn:Takasu}
		\begin{tikzcd}
			\coprod_{P\in \calP}G\times_P EP\ar{r}\ar{d}
			& EG\ar{d}
			\\
			\coprod_{P\in \calP} G/P\ar{r}
			& Z
		\end{tikzcd}
		\end{equation}
		where the upper horizontal map is the inclusion of a $G$-subcomplex and the left vertical map is induced by the constant map~$EP\to P/P$ for every~$P\in \calP$.
		Note that~$Z^\sing=G/\calP$.
		Recall that a $G$-CW-pair~$(X,A)$ is relatively free if every cell in~$X\smallsetminus A$ has trivial stabiliser.
		
		(i) Suppose that~$(G,\calP)$ is of type~$\sfF_n$ and let~$(X,G/\calP)$ be a relatively free $G$-CW-pair, where~$X$ is contractible and has cocompact $n$-skeleton.
		Since the collection~$\calP$ is malnormal and consists of distinct subgroups, the $G$-CW-complex~$X$ is a model for~$E_\calFP G$ with cocompact $n$-skeleton.
		
		Suppose that~$G$ is of type~$\calFP$-$\sfF_n$ and let~$Y$ be a $G$-CW-model for~$E_\calFP G$ with cocompact $n$-skeleton.
		The classifying $G$-map~$Y\to Z$, where~$Z$ is the model for~$\EFPG$ from~\eqref{eqn:Takasu}, is a $G$-homotopy equivalence and restricts to a $G$-homotopy equivalence $Y^\sing\to Z^\sing=G/\calP$.
		We denote by~$\hat Y$ the following $G$-pushout
		\[\begin{tikzcd}
			Y^\sing\ar{r}\ar{d}
			& Y\ar{d} \\
			G/\calP\ar{r}
			& \hat Y
		\end{tikzcd}\]
		where the upper map is the inclusion~$Y^\sing\to Y$ and the left map is the $G$-homotopy equivalence~$Y^\sing\to G/\calP$.
		Then the $G$-CW-pair~$(\hat Y,G/\calP)$ is relatively free, and~$\hat Y$ is contractible and has cocompact $n$-skeleton. 
				
		(ii) Suppose that~$(G,\calP)$ is of type~$\sfFP_n$ and let~$L_*$ be a free $\IZ G$-resolution of~$\Delta_{G/\calP}$ that is finitely generated in degrees~$\le n-1$.
		By composing the augmentation map~$L_0\to \Delta_{G/\calP}$ with the inclusion~$\Delta_{G/\calP}\to \IZ[G/\calP]$, we obtain a $\IZ G$-chain complex
		\begin{equation}
		\label{eqn:ZG_to_Bredon}
			\cdots\to L_1\to L_0\to \IZ[G/\calP]
		\end{equation}
		which is a $\IZ G$-resolution of~$\IZ$.
		We define an $\OFPG$-chain complex~$\underline{L}_*$ as follows:
		Set~$\underline{L}_*(G/1)$ to be the chain complex~\eqref{eqn:ZG_to_Bredon}.
		For~$K\in \calFP$ with~$K\neq 1$, $K\subset gPg^{-1}$ for~$g\in G$ and~$P\in \calP$, set~$\underline{L}_*(G/K)$ to be the $\IZ$-subcomplex of~$\underline{L}_*(G/1)$ that is concentrated in degree~$0$ given by the $\IZ$-summand of~$\IZ[G/\calP]$ corresponding to~$gP\in G/\calP$.
		This is well-defined, since the collection~$\calP$ is malnormal and consists of distinct subgroups.
		Then~$\underline{L}_*$ is a free $\OFPG$-resolution of the trivial $\OFPG$-module~$\underline{\IZ}$ which is finitely generated in degrees~$\le n$.
		
		Suppose that~$G$ is of type~$\calFP$-$\sfFP_n$ and let~$\underline{M}_*$ be a free $\OFPG$-resolution of the trivial~$\OFPG$-module~$\underline{\IZ}$ that is finitely generated in degrees~$\le n$.
		The cellular~$\OFPG$-chain complex~$\underline{C}_*(Z)$ of the model~$Z$ for~$\EFPG$ from~\eqref{eqn:Takasu} is a free $\OFPG$-resolution of~$\underline{\IZ}$.
		By the fundamental lemma of homological algebra, there exists an $\OFPG$-chain homotopy equivalence~$\underline{M}_*\to \underline{C}_*(Z)$ and it restricts to an $\OFPG$-chain homotopy equivalence~$\underline{M}_*^\sing\to \underline{C}_*(Z)^\sing$.
		Note that~$\underline{C}_*(Z)^\sing\cong \underline{C}_*(Z^\sing)=\underline{C}_*(G/\calP)$.
		We define the $\OFPG$-chain complex
		\[
			\hat{\underline{M}}_*\coloneqq \operatorname{coker}\bigl(\underline{M}_*^\sing\to \underline{M}_*\oplus \underline{C}_*(G/\calP)\bigr),
		\]
		where the map is the sum of the inclusion~$\underline{M}_*^\sing\to \underline{M}_*$ and the $\OFPG$-chain homotopy equivalence~$\underline{M}_*^\sing\to \underline{C}_*(G/\calP)$.
		Then~$\hat{\underline{M}}_*$ is a free $\OFPG$-resolution of~$\underline{\IZ}$ which is finitely generated in degrees~$\le n$ and satisfies~$\hat{\underline{M}}_*^\sing\cong \underline{C}_*(G/\calP)$.
		By evaluating~$\hat{\underline{M}}_*$ at~$G/1$, we obtain a $\IZ G$-resolution of~$\IZ$ of the form
		\[
			\cdots\to \hat{\underline{M}}_2(G/1)\to \hat{\underline{M}}_1(G/1)\to L\oplus \IZ[G/\calP]\to \IZ\to 0,
		\]
		where the $\IZ G$-modules~$L$ and~$\hat{\underline{M}}_i(G/1)$ for~$i=1,\ldots,n$ are finitely generated free.
		Hence the group pair~$(G,\calP)$ is of type~$\sfFP_n$ by Lemma~\ref{lem:0-unfree_resolutions}.
	\end{proof}	
\end{prop}

As an immediate consequence of Proposition~\ref{prop:Bredon_malnormal}, Theorem~\ref{thm:FPn_QI} and Corollary~\ref{cor:Fn_QI}, it follows that the Bredon finiteness properties for malnormal collections are inherited by quasi-retracts.

\begin{cor}
\label{cor:Bredon_QI}
	Let~$(G,\calP)$ and~$(H,\calQ)$ be group pairs, where~$\calP$ and~$\calQ$ are malnormal collections consisting of distinct subgroups.
	Let~$n\in \IN$ with~$n\ge 2$.
	Assume that~$(G,\calP)$ is a quasi-retract of~$(H,\calQ)$.
	Then the following hold:
	\begin{enumerate}[label=\enum]
		\item If~$H$ is of type~$\calFQ$-$\sfFP_n$, then~$G$ is of type~$\calFP$-$\sfFP_n$;
		\item If~$H$ is of type~$\calFQ$-$\sfF_n$, then~$G$ is of type~$\calFP$-$\sfF_n$.
	\end{enumerate}
\end{cor}

\bibliographystyle{alpha}
\bibliography{bib}

\setlength{\parindent}{0cm}

\end{document}